\documentclass[letter]{amsart}
\usepackage{amsmath,amsfonts,amsthm,amssymb}
\usepackage{hyperref} 
\usepackage{graphicx}
\usepackage{epstopdf}

\theoremstyle{plain}
\newtheorem{theorem}{Theorem}[section]

\newtheorem{lemma}[theorem]{Lemma}
\newtheorem*{de-lemma}{Lemma}
\newtheorem{proposition}[theorem]{Proposition}
\newtheorem{corollary}[theorem]{Corollary}

\theoremstyle{remark}

\theoremstyle{definition}

\newcommand{\dd}{\mathrm{d}}
\newcommand{\R}{\mathbb{R}}

\begin{document}

\title{Minimal heteroclinics for a class of fourth order O.D.E. systems}

\author{Panayotis Smyrnelis}
\address{Centro
de Modelamiento Matem\'atico (UMI 2807 CNRS), Universidad de Chile, Casilla
170 Correo 3, Santiago, Chile.}
\email{psmyrnelis@dim.uchile.cl}
\thanks{P. Smyrnelis was partially supported by Fondo Basal CMM-Chile and Fondecyt postdoctoral grant 3160055}

\date{}

\maketitle
\begin{abstract}
We prove the existence of minimal heteroclinic orbits for a class of fourth order O.D.E. systems with variational structure. In our general set-up, the set of equilibria of these systems is a union of manifolds, and the heteroclinic orbits connect two disjoint components of this set.

\end{abstract}

\section{Introduction and main results}

Given a smooth nonnegative function $W:\R^m\times \R^m\to [0,\infty)$ ($m \geq 1$), we consider the system:
\begin{equation}\label{scalar}
\frac{\dd^4 u}{\dd x^4}+W_u(u, u')-W_{uv}(u, u') u'-W_{vv}(u, u') u''=0, \ u:\R \to \R^m, 
\end{equation}
which is the Euler-Lagrange equation of the energy functional: 
\begin{equation}\label{energy}
J_{\R}(u)=\int_{\R}\Big(\frac{1}{2}| u''|^2+W(u,u')\Big), \ u\in W^{2,2}_{\mathrm{loc}}(\R;\R^m). 
\end{equation}

In the scalar case $(m=1)$, setting $W(u,v)=\frac{1}{4}(u^2-1)^2+\frac{\beta}{2}v^2$, where $\beta>0$\footnote{In the case where $\beta<0$, the corresponding equation is known as the Swift-Hohenberg equation.}, 
we obtain the Extended Fisher-Kolmogorov equation
\begin{equation}\label{efk}
\frac{\dd^4 u}{\dd x^4}-\beta u''+u^3-u=0, \ u:\R\to\R,
\end{equation}
which was proposed in 1988 by Dee and van Saarloos \cite{dee} as a higher-order model equation for bistable systems.
Equation \eqref{efk} has been extensively studied by different methods: topological shooting methods, Hamiltonian methods, variational methods, and
methods based on the maximum principle (cf. \cite{bon1}, \cite{pel1}, and the references therein, in particular \cite{pel0}, \cite{pel00}, \cite{pel000}, and \cite{pel0000}).
In recent years, it has become evident that the
structure of solutions of \eqref{efk} is considerably richer than the structure of solutions 
of the Allen-Cahn O.D.E.:
\begin{equation}\label{accc}
 u''=u^3-u, \ u:\R\to\R,  
\end{equation}
or equivalently $u''=W'(u)$, with $W(u)=\frac{1}{4}(u^2-1)^2$. Depending on the value of $\beta$, we mention below some properties of the heteroclinic orbits\footnote{The existence of heteroclinic solutions of \eqref{efk} via variational arguments was investigated for the first time by L. A. Peletier, W. C. Troy
and R. C. A. M. VanderVorst \cite{pel2}, and W. D. Kalies, R. C. A. M. VanderVorst \cite{kal}.}
of \eqref{efk}, connecting at $\pm\infty$ the two equilibria $\pm 1$, in the sense that
\begin{equation}\label{hettc}
 \lim_{x\to\pm\infty}(u(x),u'(x),u''(x),u'''(x))=(\pm 1,0,0,0) \text{ in the phase-space}.
\end{equation}
When $\beta\geq\sqrt{8}$,\footnote{The linearization of \eqref{efk} at $\pm 1$ reads $\frac{\dd^4 v}{\dd x^4}-\beta v''+2v=0$. The four roots of the associated characterictic equation $\lambda^4-\beta\lambda^2 +2=0$ are all real if and only if $\beta\geq\sqrt{8}$.} the structure of bounded solutions of \eqref{efk} exactly mirrors that of \eqref{accc}. In particular,
\eqref{efk} has (up to translations) a unique heteroclinic orbit connecting $-1$ to $1$, which is monotone. However, as soon as $\beta$ passes the critical value $\sqrt{8}$ from above, an infinity of heteroclinics appears immediately, and these orbits are no longer monotone. Actually, they oscillate around the equilibria $\pm1$, and may jump from $-1$ to $1$ and back a number of times. Also note that as $\beta$ decreases from $\sqrt{8}$, these orbits continue
to exist up to $\beta = 0$, and even somewhat beyond. 

Another major difference between the second order model \eqref{accc} and \eqref{efk}, lies in the existence of {\em pulses} for $\beta<\sqrt{8}$, i.e. nontrivial solutions $u:\R\to\R$ of \eqref{efk} such that 
\begin{equation}\label{pulse}
 \lim_{|x|\to\infty}(u(x),u'(x),u''(x),u'''(x))=(1,0,0,0) \text{ or } (-1,0,0,0).
\end{equation}
This situation which is excluded for the scalar equation \eqref{accc}, may occur if we consider the system $u''=\nabla W(u)$ with a multiple well potential $W:\R^2\to[0,\infty)$ (cf. \cite[Remark 2.6]{antonop} and \cite[Section 2]{ps}).

A more general version of the canonical equation \eqref{efk} is given by 
\begin{equation}\label{efkgen}
\frac{\dd^4 u}{\dd x^4}-g(u)u''-\frac{g'(u)}{2}(u')^2+f'(u)=0, \ u:\R\to\R, \ W(u,v)=\frac{g(u)}{2}v^2+f(u), 
\end{equation}
where $f:\R\to\R$, and $g:\R\to\R$, are smooth functions (cf. \cite{bon0}, \cite{bon00}). For instance in \cite{bon00}, a double well potential $f\geq 0$ is considered, and $g$ is allowed to take negative values to an extent that is balanced by $f$. Provided that $\inf g$ is bigger than a negative constant depending on the nondegeneracy of the minima of $f$, the variational method can be applied to construct heteroclinics of \eqref{efkgen}.

The scope of this paper is to establish the existence of minimal heteroclinics for system \eqref{scalar} in a general set-up, similar to that considered in \cite{antonop} for the Hamiltonian system $u''=\nabla W(u)$.
In particular, we allow the function $W$ to vanish on submanifolds, and we are interested in connecting two disjoint subsets of minima of $W$.

We assume that $W\in C^2(\R^m\times\R^m;[0,\infty))$ is a nonnegative function such that
\begin{description}
\item[$\mathbf{H}_1$] The set $A:=\{u\in\R^m: W(u,0)=0\}$ is partitioned into two nonempty disjoint compact subsets $A^-$ and $A^+$.
\item[$\mathbf{H}_2$] There exists an open set $\Omega\subset\R^m$ such that
$A^-\subset\Omega$, $A^+\cap\overline \Omega=\emptyset$, and 
$W(u,v)>0$ holds for every $u\in\partial\Omega$,  and for every $v\in\R^m$.
\item[$\mathbf{H}_3$] $\liminf_{\vert u\vert\rightarrow\infty}W(u,v)>0$, uniformly in $v\in\R^m$.
\end{description}
In $\mathbf{H}_1$, we define the sets $A^-$ and $A^+$ that we are going to connect. On the other hand, Hypothesis $\mathbf{H}_2$ ensures that the energy required to connect a neighbourhood of $A^-$ to a neighbourhood of $A^+$ cannot become arbitrarily small. As a consequence an orbit with finite energy may travel from $A^-$ to $A^+$ and back, only a finite number of times (cf. Lemma \ref{l44}). Also note that $W$ is allowed to vanish if $u\notin\partial\Omega$, and $v\neq 0$. Finally, Hypothesis $\mathbf{H}_3$ is assumed to derive the boundedness of finite energy orbits (cf. Lemma \ref{uniff}).

Some typical examples of functions satisfying $\mathbf{H}_i$, $i=1,2,3$, 
are given by $W(u,v)=F(u)$, $W(u,v)=F(u)+\frac{\beta}{2}|v|^2$ (vector analog of \eqref{efk}), $W(u,v)=F(u)+\frac{G(u)}{2}|v|^2$ (vector analog of \eqref{efkgen}), 
where $F:\R^m\to[0,\infty)$ is a multiple well potential such that  $\liminf_{\vert u\vert\rightarrow\infty}F(u)>0$, $G:\R^m\to[0,\infty)$, and $\beta>0$.
In particular, our results apply to the system
\begin{equation}\label{bihar}
\frac{\dd^4 u}{\dd x^4}+\nabla F(u)=0, \ u:\R\to\R^m,
\end{equation}
to the vector Extended Fisher-Kolmogorov equation
\begin{equation}\label{efkv}
\frac{\dd^4 u}{\dd x^4}-\beta u''+\nabla F(u)=0, \ u:\R\to\R^m,\ \beta>0,
\end{equation}
or to
\begin{equation}\label{efkgenv}
\frac{\dd^4 u}{\dd x^4}-G(u)u''+\frac{\nabla G(u)}{2}|u'|^2+\nabla F(u)-(\nabla G(u)\cdot u')u'=0, \ u:\R\to\R^m. 
\end{equation}
Let $q \in \big(0,\frac{d(A^-,A^+)}{2}\big)$, be such that $$\{ u\in\R^m: d(u,A^-)\leq q\}\subset\Omega, \text{ and } \{ u\in\R^m: d(u,A^+)\leq q\}\cap\overline \Omega=\emptyset.$$ We define the class $\mathcal{A}$ by:
\[\mathcal{A}=\Big\{u\in W_{\rm loc}^{2,2}(\R;\R^{m}):\left.  \begin{array}{l} d( u(x), A^-) \leq q,\text{ for }x\leq x_u^-,\\
d( u(x), A^+)\leq q,\text{ for }x\geq x_u^+,\end{array}\right.\text{ for some }x_u^-<x_u^+\Big\},\]
where $d$ stands for the Euclidean distance. Note that no limitation is imposed on the numbers $x_u^-<x_u^+$ that may largely depend on $u$.
Our main theorem establishes the existence of a connecting minimizer in the class $\mathcal{A}$:
\begin{theorem}\label{conn-exists}
Assume $W$ satisfies $\mathbf{H}_i$, $i=1,2,3$. Then $J_{\R}(u)$ admits a minimizer $\bar{u}\in\mathcal{A}$:
\[J_{\R}(\bar{u})=\min_{u\in\mathcal{A}}J_{\R}(u)<+\infty.\]
Moreover it results that\footnote{The Existence of a minimizer $\bar u$ satisfying (ii) is ensured provided that $W$ is continuous 
(cf. the proof in Section \ref{sec:pfff}). On the other hand, the $C^1$ smoothness of $W$ and $W_u$ is required to establish properties (i), (iii) and (iv). }
\begin{itemize}
\item[(i)] $\bar u \in C^4(\R;\R^m)$ solves \eqref{scalar}
\item[(ii)]$\lim_{x\rightarrow \pm\infty}d(\bar{u}(x), A^\pm)=0$,
\item[(iii)] $\lim_{x\rightarrow \pm\infty}(\bar{u}'(x), \bar{u}''(x),\bar{u}'''(x))=(0,0,0)$,
\item[(iv)]$H:=\frac{1}{2}| \bar u''|^2-W(\bar u,\bar u')+W_v(\bar u,\bar u')\cdot\bar u'-\bar u'''\cdot \bar u'\equiv0$, $\forall x \in \R$. 
\end{itemize}
\end{theorem}
An immediate consequence of Theorem \ref{conn-exists} is
\begin{corollary}\label{main2}
Assume that $A=\{a_1,\ldots, a_N\}$ for some $N\geq 2$, and given $a^-\in A$, set $A^-=\{a^-\}$ and $A^+= A\setminus\{a^-\}$.
Then under the assumptions of Theorem \ref{conn-exists}, there exists $a^+\in A^+$ such that the minimizer $\bar u$ satisfies
$ \lim_{x\to \pm\infty}\bar u(x)=a^\pm$.
\end{corollary}
By construction, the minimizer $\bar u$ of Theorem \ref{conn-exists} is a minimal solution of (\ref{scalar}), in the sense that
\[
J_{\mathrm{supp}\, \phi}(\bar u)\leq J_{\mathrm{supp}\, \phi}(\bar u+\phi)
\]
for all $\phi\in C^\infty_0(\R;\R^m)$. This notion of minimality is standard for many problems in which the energy of a localized solution is actually infinite due to non compactness of the domain. 
The {\em Hamiltonian} $H$ introduced in property (iv) of Theorem \ref{conn-exists}, is a constant function for every solution of \eqref{scalar}. In the case of system $u''=\nabla W(u)$, we have $H=\frac{1}{2}|u'|^2-W(u)$, and every heteroclinic orbit satisfies the equipartition relation $H=0\Leftrightarrow\frac{1}{2}|u'|^2=W(u)$. We also point out that in the general set-up of Theorem \ref{conn-exists}, the minimizer $\bar u$ is a heteroclinic orbit only in a weak sense, since $\bar u$ approaches the sets $A^\pm$ at $\pm\infty$, but the limits of $\bar u$ at $\pm\infty$ may not exist. In Section \ref{sec:cvvv}, we will study the asymptotic convergence of $\bar u$, and establish an exponential estimate under a convexity assumption on $W$ in a neighbourhood of the smooth orientable surfaces $A^\pm$. From this estimate, it follows that the limits of $\bar u$ exist at $\pm \infty$. As a consequence, in many standard situations, the orbit of $\bar u$ actually connects two points $a^\pm\in A^\pm$.  

The next Section contains the proof of Theorem \ref{conn-exists}. 
In contrast with \cite{antonop}, we avoid utilizing comparison arguments, since this method applied to higher order problems requires a lot of calculation. Indeed,
to modify $W^{2,2}$ Sobolev maps, we also have to ensure the continuity of the first derivatives. Two ideas in Lemma \ref{l44} are crucial in the proof of Theorem \ref{conn-exists}. Firstly, the fact that a finite energy orbit may travel from $A^-$ to $A^+$ and back, only a finite number of times in view of $\mathbf{H}_2$. Secondly, an inductive argument to consider appropriate translations of the minimizing sequence, and fix the loss of compactness issue due to the translation invariance of \eqref{scalar}.

\section{Proof of Theorem \ref{conn-exists}}\label{sec:pfff}
We first establish the following Lemmas:
\begin{lemma}\label{lem1}
There exists $u_0\in\mathcal{A}$ satisfying
\begin{equation}\label{j-bounded}
 J_\R(u_0)<+\infty.\end{equation}
\end{lemma}
\begin{proof}
Indeed, let $a^\pm \in  A^\pm$ be such that $d(a^-,a^+)=d(A^-,A^+)$. We define
\begin{equation*}
u_0(x)=\begin{cases}
a^-,&\text{ for } x\leq 0,\\
a^-+(2x^2-x^4)(a^+ -a^-),&\text{ for }0 \leq x\leq 1,\\
a^+,&\text{ for } x\geq 1,
\end{cases}
 \end{equation*}
which clearly belongs to $\mathcal A$ and satisfies \eqref{j-bounded}.
\end{proof}
 From (\ref{j-bounded}) it follows that
 \[\inf_{u\in\mathcal{A}}J_{\R}(u)=\inf_{u\in\mathcal{A}_b}J_{\R}(u)<+\infty,\]
 where
 \[\mathcal{A}_b=\{ u\in\mathcal{A}:\ J_{\R}(u)\leq J_{\R}(u_0)\}.\]

\begin{lemma}\label{uniff}
The maps $u\in \mathcal A_b$ and their first derivatives are uniformly bounded.
In addition, the derivatives $u'$ of the maps $u\in\mathcal A_b$ are equicontinuous.
\end{lemma}
\begin{proof}
We first notice that the first derivative of a map $u\in\mathcal A_b$ is H\"{o}lder continuous, since by the Cauchy-Schwarz inequality we have
\begin{equation}\label{css0}
|u'(y)-u'(x)|\leq \Big(\int_{x}^{y}|u''|^2\Big)^{1/2}\sqrt{y-x}\leq \sqrt{2 J_\R(u_0)}\sqrt{y-x} , \ \forall x<y.
\end{equation}
This proves that the derivatives $u'$ of the maps $u\in\mathcal A_b$ are equicontinuous.

Next, we establish the uniform boundedness of the maps $u\in \mathcal A_b$. Let $R>0$ be large enough and such that $d(u, A^-\cup A^+)\leq q$ implies that $|u|<R$. According to Hypothesis $\mathbf{H}_3$, we can find a constant $w_R>0$ such that 
$W(u,v)\geq w_R$, for every $u \in \R^m$ such that $|u|\geq R$, and for every $v\in\R^m$. It follows that for every map $u\in \mathcal A_b$ we have 
$w_R\mathcal L^1(\{x\in \R: |u(x)|\geq R\})\leq\int_\R W(u)\leq  J_{\R}(u_0)$, where $\mathcal L^1$ denotes the one dimensional Lebesgue measure. As a consequence, if $u$ takes a value $u(x_2)=L \nu $ with $L>R$ and $\nu$ a unit vector, we can find an interval $x_1<x_2$ such that $|u(x_1)|=R$ and $|u(x)|\geq R$, $\forall t\in[x_1,x_2]$. Then, we have
$L-R\leq \int_{x_1}^{x_2} u'(x)\cdot \nu \, \dd x$, and this implies the existence of $y_1\in [x_1,x_2]$ such that 
such that $u'(y_1)\cdot \nu\geq  \frac{(L-R)}{x_2-x_1} \geq \frac{(L-R)w_R}{J_\R(u_0)}$. Similarly, we can find $x_3>x_2$ such that $|u(x_3)|=R$ and $|u(x)|\geq R$, $\forall t\in[x_2,x_3]$. As previously, there exists
$y_3\in[ x_2,x_3]$ such that $ u'(y_3)\cdot \nu \leq -\frac{(L-R)w_R}{J_\R(u_0)},$ and by construction $y_3-y_1\leq \frac{J_\R(u_0)}{w_R}$. Finally in view of \eqref{css0} we obtain
\begin{equation*}\label{css}
\frac{2(L-R)w_R}{J_\R(u_0)}\leq|u'(y_3)-u'(y_1)|\leq \sqrt{2 J_\R(u_0)}\sqrt{y_3-y_1}\leq J_\R(u_0)\sqrt{\frac{2}{w_R}},
\end{equation*}
and deduce that $L\leq M:= R+\frac{1}{\sqrt{2}}\frac{(J_\R(u_0))^2}{w_R^{3/2}}$, which proves the uniform bound for $u\in \mathcal A_b$. 

Now, suppose that $u'(x_0)=\Lambda\nu$
with $\Lambda >\sqrt{2J_\R(u_0)}$ and $\nu$ a unit vector.
Utilizing again \eqref{css0} we have
$u'(x)\cdot \nu \geq \Lambda -\sqrt{2J_\R(u_0)}$ for $x\in [x_0-1,x_0+1]$.
In particular since $\|u\|_{L^\infty}\leq M$, we conclude that
$2M\geq \int_{x_0-1}^{x_0+1}(u'(x)\cdot \nu)\dd x \geq 2(\Lambda -\sqrt{2J_\R(u_0)})$ which implies that
$\Lambda \leq M+\sqrt{2J_\R(u_0)})$. This completes the proof of Lemma \ref{uniff}.
\end{proof}
 
\begin{lemma}\label{limmm}
Let $u \in W_{\rm loc}^{2,2}(\R;\R^m)$ be such that $J_{\R}(u)\leq J_{\R}(u_0)$, and $u$ as well as $u'$ are bounded, and uniformly continuous. 
Then, 
\begin{equation}\label{limmm1}
\lim_{x\to\pm\infty}d(u(x),A^-\cup A^+)=0, \text{ and } \lim_{x\to\pm\infty}u'(x)=0.
\end{equation}
\end{lemma}
\begin{proof}
We first assume by contradiction that $\lim_{x\to\pm\infty}u'(x)=0$ does not hold. Without loss of generality, we consider a sequence $\{x_k\}$ such that
$\lim_{k\to\infty}x_k=+\infty$, 
and $\lim_{k\to\infty}u'(x_k)=\lambda \nu$, with $\lambda\neq 0$, and $\nu$ a unit vector.
Let $k_0$ be large enough, such that $u'(x_k)\cdot\nu\geq \frac{3\lambda}{4}$ for every $k\geq k_0$, and let $I_k=[a_k,b_k]$ be the largest interval containing $x_k$ and such that $u'\cdot\nu\geq\lambda/2$ holds on $I_k$. Since $\|u\|_{L^\infty}\leq M$, it is clear that $\mathcal L^1(I_k)\leq\frac{4M}{\lambda}$, where $\mathcal L^1$ denotes the one dimensional Lebesgue measure. Moreover, we have $u'(a_k)\cdot\nu=\lambda/2$. Applying \eqref{css0} in the interval $[a_k,x_k]$, it follows that
$\frac{\lambda^{3}}{4^3M}\leq \int_{I_k} |u''|^2$. Since, by passing to a subsequence if necessary, we can assume that the intervals $I_k$ are disjoint, this contradicts $J_{\R}(u)\leq J_{\R}(u_0)$.

Next, we assume by contradiction that $\lim_{x\to\pm\infty}d(u(x),A^-\cup A^+)=0$ does not hold. Without loss of generality, we consider a sequence $\{x_k\}$ such that
$\lim_{k\to\infty}x_k=+\infty$, 
$\lim_{k\to\infty}u(x_k)=l\notin A^- \cup A^+$, 
and $\lim_{k\to\infty}u'(x_k)=0$. 
Since $u$ as well as $u'$ are bounded and uniformly continuous, the function $x\to W(u(x),u'(x))$ is also uniformly continuous. 
In view of $\mathbf{H}_1$, there exists $\delta>0$ independent of $k$ such that
$W(u(x),u'(x))\geq W(l,0)/2>0$, for every $ x \in [x_k-\delta, x_k+\delta]$, and $k\geq k_0$ large enough. In particular we have $J_{[x_k-\delta,x_k+\delta]}(u)\geq \delta W(l,0)$, for $k=k_0, k_0+1,\ldots$.
Since, by passing to a subsequence if necessary, we can assume that the intervals $[x_k-\delta,x_k+\delta]$ are disjoint, we reach again a contradiction.
\end{proof}

\begin{lemma}\label{l44}
There exists $\bar u \in \mathcal A_b$ satisfying $J_{\R}(\bar{u})=\min_{u\in\mathcal{A}_b}J_{\R}(u)<+\infty$, and property (ii) of Theorem \ref{conn-exists}.
\end{lemma}
\begin{proof}
We consider a sequence $u_k\in\mathcal A_b$ such that $\lim_{k\to\infty}J(u_k)=\inf_{u\in\mathcal{A}_b}J_{\R}(u)$.
For every $k$ we define the sequence $$-\infty<x_1(k)<x_2(k)<\ldots<x_{2N_k-1}(k)<x_{2N_k}(k)=\infty$$
by induction:
\begin{itemize}
\item $x_1(k)=\sup\{ t\in\R: \, d(u_k(s),A^+)\geq q, \forall s\leq t\}<\infty$,
\item $x_{2i}(k)=\sup\{ t\in\R: \, d(u_k(s),A^-)\geq q, \forall s\in [x_{2i-1}(k),t]\}\leq\infty$,
\item $x_{2i+1}(k)=\sup\{ t\in\R: \, d(u_k(s),A^+)\geq q, \forall s\in [x_{2i}(k),t]\}<\infty$, if $x_{2i}(k)<\infty$,
\end{itemize}
where $i=1,\ldots$.
In addition, we set
\begin{itemize}
\item $y_{2i-1}(k)=\sup\{ t\leq x_{2i-1}(k): \, d(u_k(t),A^-)\leq q\}$, 
\item $y_{2i}(k)=\sup\{ t\leq x_{2i}(k): \, d(u_k(t),A^+)\leq q\}$, if $x_{2i}(k)<\infty$.
\end{itemize}

\begin{figure}[h]
\begin{center}
\includegraphics[scale=1]{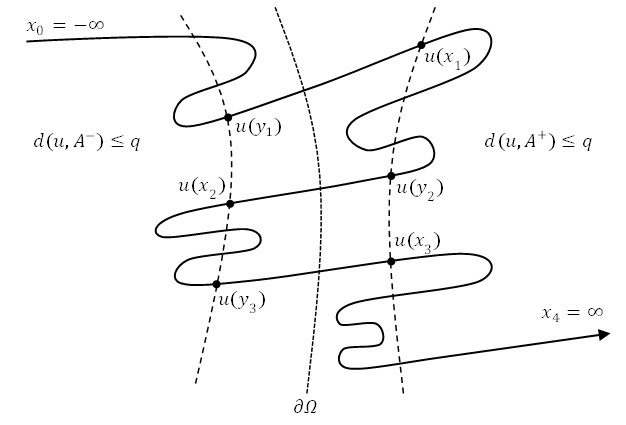}
\end{center}
\caption{The sequence $-\infty=x_0<y_1<x_1<y_2<x_2<\ldots<x_{2N}=\infty$, ($N=2$).} 
\label{fig}
\end{figure}
By Lemma \ref{uniff}, we have the uniform bounds $M:=\sup_k\|u_k\|_{L^\infty}<\infty$, and $\Lambda:=\sup_k\|u'_k\|_{L^\infty}<\infty$.
Let $\delta>0$ be such that
\begin{itemize}
\item $B_\delta(z)\cap\{ u\in\R^m: d(u,A^-)\leq q\}=\emptyset$, and $B_\delta(z)\cap\{ u\in\R^m: d(u,A^+)\leq q\}=\emptyset$, $\forall z\in\partial\Omega\cap B_M$,
\item $W(u,v)>0$ holds on $\{(u,v)\in B_M\times B_\Lambda: d(u,\partial\Omega\cap B_M)\leq \delta\}$ (cf. Hypothesis $\mathbf{H}_2$),
\end{itemize}
where $B_R(z)\subset \R^m$ denotes the closed ball of radius $R$ centered at $z\in\R^m$, and $B_R$ the closed ball of radius $R$ centered at the origin.

Next, we notice that in every interval $[y_j(k),x_j(k)]$ ($j=1,\ldots,2N_k-1$), 
there exists $z_j(k)\in [y_j(k),x_j(k)]$ such that $u_k(z_j(k))\in\partial\Omega$. Let $I_j(k)=[a_j(k),b_j(k)]$ 
be the largest interval containing $z_j(k)$, and such that $|u_k(x)-u_k(z_j(k))|\leq \delta $, $\forall x \in I_j(k)$. Since $|u_k(a_j(k))-u_k(z_j(k))|=\delta$,
and $|u_k(b_j(k))-u_k(z_j(k))|=\delta$, it is clear that
\begin{equation}\label{we1}
2\delta\leq \int_{a_j(k)}^{b_j(k)}|u'_k|\leq \Lambda(b_j(k)-a_j(k)),\nonumber
\end{equation}
and
\begin{equation}\label{we2}
\int_{a_j(k)}^{b_j(k)}W(u_k,u'_k)\geq w_\delta (b_j(k)-a_j(k))\geq w_\delta \frac{2\delta}{ \Lambda},\nonumber
\end{equation}
where $w_\delta:=\inf \{W(u,v): d(u,\partial\Omega\cap B_M)\leq \delta,\, |u|\leq M, \, |v|\leq \Lambda\}>0$.
Since the intervals $[a_j(k),b_j(k)]\subset[y_j(k),x_j(k)]$ are disjoint for every $j=1,\ldots,2N_k-1$, it follows that
\begin{equation}\label{we3}
(2N_k-1)w_\delta \frac{2\delta}{ \Lambda}\leq\int_{\R}W(u_k, u'_k)\leq J_{\R}(u_0),\nonumber
\end{equation}
and thus the integers $N_k$ are uniformly bounded. By passing to a subsequence, we may assume that $N_k$ is a constant integer $N\geq 1$. 

Our next claim is that up to subsequence, there exist an integer $i_0$ ($1\leq i_0\leq N$) and an integer $j_0$ ($i_0\leq j_0\leq N$) such that
\begin{itemize}
\item the sequence $x_{2j_0-1}(k)-x_{2i_0-1}(k)$ is bounded,
\item $\lim_{k\to\infty}(x_{2i_0-1}(k)-x_{2i_0-2}(k))=\infty$,
\item $\lim_{k\to\infty}(x_{2j_0}(k)-x_{2j_0-1}(k))=\infty$,
\end{itemize}
where for convenience we have set $x_0(k):=-\infty$. 

Indeed, we are going to prove by induction on $N\geq 1$, that given $2N+1$ sequences $-\infty\leq x_0(k)<x_1(k)<\ldots<x_{2N}(k)\leq\infty$, such that $\lim_{k\to\infty}(x_1(k)-x_0(k))=\infty$, and $\lim_{k\to\infty}(x_{2N}(k)-x_{2N-1}(k))=\infty$, then up to subsequence the properties (a), (b), and (c) above hold, for two fixed indices $1\leq i_0\leq j_0\leq N$.
When $N=1$, the assumption holds by taking $i_0=j_0=1$. Assume now that $N>1$, and let $l \geq 1$ be the largest integer such that
the sequence $x_{l}(k)-x_{1}(k)$ is bounded. Note that $l<2N$. If $l$ is odd, we are done, since the sequence $x_{l+1}(k)-x_{l}(k)$ is unbounded, 
and thus we can extract a subsequence $\{n_k\}$ such that $\lim_{k\to\infty}(x_{l+1}(n_k)-x_{l}(n_k))=\infty$. Otherwise $l=2m$ (with $1\leq m<N$), 
and the sequence $x_{2m+1}(k)-x_{2m}(k)$ is unbounded. We extract a subsequence $\{n_k\}$ such that 
$\lim_{k\to\infty}(x_{2m+1}(n_k)-x_{2m}(n_k))=\infty$.
Then, we apply the inductive statement with $N'=N-m$, to the $2N'+1$ sequences $x_{2m}(n_k)<x_{2m+1}(n_k)<\ldots<x_{2N}(n_k)$.

At this stage, we consider appropriate translations of the sequence $\{u_k\}$, by setting $\bar u_k(x)=u_k(x-x_{2i_0-1}(k))$. It is obvious that $\{\bar u_k\}$ is still a minimizing sequence.
In view of Lemma \ref{uniff} we obtain by the theorem of Ascoli via a diagonal argument that
$\lim_{k\to\infty }\bar u_k=\bar u$ in $C^1_{\mathrm{loc}}$ (up to subsequence).
On the other hand, since $\int_{\R}|\bar u''_k|^2\leq 2 J_{\R}(u_0)$ we deduce that $\bar u''_k\rightharpoonup \bar v$ 
weakly in $L^2(\R;\R^m)$ (up to subsequence). One can check that actually $\bar u \in W_{\rm loc}^{2,2}(\R;\R^m)$, and $\bar u''= \bar v$.
Finally, we have by lower semicontinuity
\begin{equation}\label{la1}
\int_{\R}|\bar u''|^2\leq\liminf_{k\to\infty} \int_{\R}|\bar u''_k|^2,
\end{equation}
and by Fatou's Lemma
\begin{equation}\label{la2}
\int_{\R}W(\bar u,\bar u')\leq\liminf_{k\to\infty} \int_{\R}W(\bar u_k,\bar u'_k). 
\end{equation}
It follows from \eqref{la1} and \eqref{la2} that
$J_{\R}(\bar{u})\leq\inf_{u\in\mathcal{A}_b}J_{\R}(u)<+\infty$. To complete the proof it remains to show that $\bar u \in \mathcal A$. 
Indeed, in the interval $[x_{2i_0-2}(k),x_{2i_0-1}(k)]$ we have $d(u_k(x),A^+)\geq q$, thus since $\lim_{k\to\infty}(x_{2i_0-1}(k)-x_{2i_0-2}(k))=\infty$, we deduce that $d(\bar u(x), A^+)\geq q$, for $x \leq 0$. Similarly,
in the interval $[x_{2j_0-1}(k),x_{2j_0}(k)]$ we have $d(u_k(x),A^-)\geq q$, thus since $\lim_{k\to\infty}(x_{2j_0}(k)-x_{2j_0-1}(k))=\infty$, 
while the sequence $x_{2j_0-1}(k)-x_{2i_0-1}(k)$ is bounded, it follows that $d(\bar u(x), A^-)\geq q$, in a neighbourhood of $+\infty$. To conclude, Lemma \ref{limmm} applied to $\bar u$, implies that $\lim_{x\to\pm\infty}d(\bar u(x),A^\pm)=0$, and thus $\bar u \in\mathcal A$.
\end{proof}
Now, we complete the proof of Theorem \ref{conn-exists}. By definition of the class $\mathcal A$, 
for every $\phi\in C^\infty_0(\R;\R^m)$, we have $\bar u+\phi\in\mathcal A$. Thus, the minimizer $\bar u$ satisfies the Euler-Lagrange equation:
\begin{equation}
\int_{\R}\Big(\bar u''\cdot\phi''+W_u(\bar u,\bar u')\cdot \phi+W_v(\bar u,\bar u')\cdot \phi'\Big)=0,\ \forall\phi\in C^\infty_0(\R;\R^m).
\end{equation}
This is the weak formulation of \eqref{scalar}. Since $W\in C^2(\R^m\times\R^m;\R)$, it follows that $\bar u \in C^4(\R;\R^m)$, and $\bar u$ is a classical solution of system \eqref{scalar}.

Next we establish property (iii). The limit $\lim_{x\rightarrow \pm\infty}\bar{u}'(x)=0$, is a consequence of Lemma \ref{limmm}. To see that $\lim_{x\rightarrow \pm\infty}\bar u''(x)=0$, we recall the interpolation inequality
\begin{equation}\label{interpol}
\int_a^{a+h}|v'|^2\leq C\Big(\int_a^{a+h}|v|^2+\int_a^{a+h}|v''|^2\Big), \ \forall v\in W^{2,2}([a,a+h];\R^m),
\end{equation}
that holds for a constant $C$ independent of $a\in\R$, and $h\in [1,\infty)$. In view of \eqref{scalar}, it is clear that
\begin{equation}\label{interpol2}
\int_a^{a+1}\Big|\frac{\dd^4 \bar u}{\dd x^4}\Big|^2\leq C_1+C_2\int_a^{a+1}|\bar u''|^2\leq C_1+2C_2 J_{\R}(\bar u)\leq C_3,
\end{equation}
where $C_i$ are constants independent of $a$. Moreover, applying \eqref{interpol} to $v=\bar u''$, we can find another constant $C_4$ independent of $a$, such that
\begin{equation}\label{interpol3}
\int_a^{a+1}|\bar u'''|^2\leq C_4,
\end{equation}
and as a consequence $\bar u''$ is uniformly continuous (see the proof of Lemma \ref{uniff}). Since $\int_{\R}|\bar u''|^2<\infty$ it follows that $\lim_{x\rightarrow \pm\infty}\bar u''(x)=0$. Finally, $\mathbf{H}_1$ and \eqref{scalar} imply that 
$\lim_{x\rightarrow \pm\infty}\frac{\dd^4 \bar u}{\dd x^4}(x)=0$. thus, we also have $\lim_{x\rightarrow \pm\infty}\bar u'''(x)=0$ (cf. \cite[\S 3.4 p. 37]{gilbarg}).

To prove property (iv), consider an arbitrary solution $u$ of \eqref{scalar}. By integrating the inner product of \eqref{scalar} by $u'$, 
one can see that the Hamiltonian $H:=\frac{1}{2}|  u''|^2-W( u, u')+W_v( u, u')\cdot  u'- u'''\cdot u'$ is constant along solutions.
In the case of the minimizer $\bar u$, the Hamiltonian is zero by properties (ii) and (iii), and by Hypothesis $\mathbf{H}_1$.

\section{Asymptotic convergence of the minimizer $\bar u$}\label{sec:cvvv}

A natural question arises in the case where the set $A^\pm$ defined in $\mathbf{H}_1$ are manifolds or union of manifolds: does the minimizer $\bar{u}$ converge to a point of $ A^+$ (respectively $A^-$) at $\pm \infty$?
Before answering this question, we are going to establish by a variational method the following exponential estimate:
\begin{proposition}\label{estimate}
Assume that $A^- \subset \R^m$ is a $C^2$ compact orientable surface with unit normal $\bf{n}$, and that $W$ satisfies 
\begin{equation}\label{nondeg}
\frac{\dd^2 W}{\dd s^2}(a+s\bf{n},s\nu)\Big|_{s=0} >0, \, \forall a \in A^-,\, \forall \nu\in \R^m \text{ such that } |\nu|=1.
\end{equation}
Then, the minimizer $\bar{u}$ constructed in Theorem \ref{conn-exists} satisfies $d(\bar{u}(x), A^-)\leq K e^{kx}$, and $|\bar{u}'(x)|\leq K e^{kx}$, $\forall x \leq 0$, for some constants $K,k>0$.
\end{proposition}
\begin{proof}
In view of \eqref{nondeg}, there exists $\lambda>0$ small enough, such that 
$\mathcal{U}:=\{u\in \R^m: d(u,A^-)<\lambda\}$ is a tubular neighbourhood  of $ A^-$ (cf \cite{docarmo}), and moreover 
\begin{equation}\label{tube}
m (d^2(u,A^-)+|v|^2)\leq W(u,v) \leq M  (d^2(u, A^-)+|v|^2), \ \forall u \in \mathcal{U}, \, \forall v\in\R^m: \, |v|<\lambda,
\end{equation}
for some constants $0<m<M$.
Let $x_0$ be such that $d(\bar{u}(x),A^-)<\lambda/8$, and $|\bar u'(x)|<\lambda/4$, $\forall x\leq x_0$. For fixed $x \leq x_0$, we set $\phi(x):=\bar u(x)-\frac{1}{2}\bar u'(x)$. One can see that $\phi(x)\in \mathcal U$, since actually 
$d(\phi(x),A^-)\leq d(\bar u(x),A^-)+\frac{1}{2}|\bar u'(x)|<\lambda/4$. We also introduce the point $a(x) \in  A^-$ such that $d(\phi(x),a(x))=d(\phi(x), A^-)$. Next we define the comparison map
\begin{equation}\label{comparison3}
z(t)=\begin{cases}
\bar{u}(x)+\big((t-x)+\frac{(t-x)^2}{2}\big)\bar u'(x) &\text{ for } x-1\leq t \leq x,\\
\phi(x)+(2(t-x+1)^2-(t-x+1)^4)(a(x)-\phi(x)) &\text{ for } x-2\leq t \leq x-1,\\
a(x) &\text{ for }  t \leq x-2.
\end{cases}\end{equation}
An easy computation shows that 
\begin{itemize}
\item[(a)] $z\in W^{2,2}_{\mathrm{loc}}((-\infty,x];\R^m)$,
\item[(b)] $z(x)=\bar u(x)$, and $z'(x)=\bar u'(x)$,
\item[(c)] $d(z(t),A^-)\leq d(\bar u(x),A^-)+\frac{1}{2}|\bar u'(x)|<\lambda/4$, $\forall t\leq x$,
\item[(d)] $|z'(t)|\leq 4d(\bar u(x),A^-)+2|\bar u'(x)|<\lambda$, $\forall t\leq x$,
\item[(e)] $J_{(-\infty,x]}(z)\leq C(d^2(\bar u(x),A^-)+|\bar u'(x)|^2)$, for a constant $C>0$ independent of $x$.
\end{itemize} 
At this stage, we set $\theta(x):=\int_{-\infty}^x (d^2(\bar{u}(t), A^-)+|\bar u'(t)|^2) \dd t$, and it is clear that $\theta'(x)=d^2(\bar{u}(x), A^-)+|\bar u'(x)|^2$.
The variational characterization of $\bar{u}$, \eqref{tube}, and (e) above, imply that
\begin{equation}\label{comput1}
m \theta(x) \leq \int_{-\infty}^x W(\bar u(t),\bar u'(t)) \dd t\leq J_{(-\infty,x]}(\bar{u})\leq J_{(-\infty,x]}(z)\leq C\theta'(x).
\end{equation}
After an integration of inequality \eqref{comput1}, we obtain that $\theta(x)\leq \theta(x_0)e^{c(x-x_0)}$, for every $x\leq x_0$, and for some constant $c>0$.
Since the functions $x\to d^2(\bar u(x), A^-)$, and $x\to |\bar u'(x)|^2$ are Lipschitz continuous (cf. Theorem \ref{conn-exists}), the statement of Proposition \ref{estimate} follows from
\begin{lemma}\label{explem}
Let $f:(-\infty,x_0]\to [0,\infty)$ be a function such that
\begin{itemize}
\item $|f(x)-f(y)|\leq M|x-y|$, $\forall x, \, y\leq x_0$,
\item $\int_{-\infty}^x f(t)\dd t\leq C e^{cx}$, $\forall x \leq x_0$,
\end{itemize}
where $M$, $c$ and $C$ are positive constants. Then, $f(x)\leq 2\sqrt{MC}e^{\frac{c}{2}x}$, $\forall x\leq x_0$.
\end{lemma}
\begin{proof}
Let $x\leq x_0$ be fixed and let $\lambda:=f(x)$.
For $t\in [x-\frac{\lambda}{2M},x]$, we have $f(t)\geq f(x)-M|t-x|\geq\frac{\lambda}{2}$. Thus,
$$\frac{\lambda^2}{4M}\leq\int_{x-\frac{\lambda}{2M}}^x f(t)\dd t\leq C e^{cx}\Rightarrow \lambda=f(x)\leq  2\sqrt{MC}e^{\frac{c}{2}x}.$$
\end{proof}
\end{proof}

\begin{corollary}\label{corlim}
Under the assumptions of Proposition \ref{estimate}, there exists $l \in A^-$ such that $\bar{u}(x) \to l$, as $x \to -\infty$.
\end{corollary}
\begin{proof}
The exponential estimates provided by Proposition \ref{estimate} imply that $x\to|\bar u('x)|$ is integrable in a neighbourhood of $-\infty$. As a consequence, it is easy to see that the limit of $\bar u$ at $-\infty$ exists and belongs to $A^-$.
\end{proof}

\begin{proposition}\label{estimate2}
Assume that $A^- =\{a^-\}$, and that $W$ satisfies 
\begin{equation}\label{isol}
m |u|^2\leq W(u,v) \text{ in a neighbourhood of $(a^-,0)$, for a constant $m>0$},
\end{equation}
Then, the minimizer $\bar{u}$ constructed in Theorem \ref{conn-exists} satisfies $|\bar{u}(x)-a^-|\leq K e^{kx}$, and $|\bar{u}'(x)|\leq K e^{kx}$, $\forall x \leq 0$, for some constants $K,k>0$.
\end{proposition}
\begin{proof}
We proceed as in the proof of Proposition \ref{estimate}. For $\lambda>0$ small enough, we have
\begin{equation}\label{tube2b}
m |u|^2\leq W(u,v) \leq M ( |u-a^-|^2+|v|^2), \ \forall u\in \R^m :\, |u-a^-|<\lambda, \, \forall v\in\R^m: \, |v|<\lambda.
\end{equation}
Let $x_0$ be such that $|\bar{u}(x)-a^-|<\lambda/8$, and $|\bar u'(x)|<\lambda/4$, $\forall x\leq x_0$. For fixed $x \leq x_0$, 
we set again $\phi(x):=\bar u(x)-\frac{1}{2}\bar u'(x)$. Next we consider the comparison map \eqref{comparison3}, where $a(x)$ is now replaced 
by $a^-$. This map $z$ still satisfies properties (a)-(e) in Proposition \ref{estimate}. Setting $\theta(x):=\int_{-\infty}^x (|\bar{u}(t)-a^-|^2+|\bar u'(t)|^2) \dd t$, we obtain
\begin{equation}\label{comput1bb}
m \int_{-\infty}^x |\bar u(t)-a^-|^2\dd t +\frac{1}{2} \int_{-\infty}^x |\bar u''(t)|^2\dd t \leq  J_{(-\infty,x]}(\bar{u})\leq J_{(-\infty,x]}(z)\leq C\theta'(x).
\end{equation}
Finally, we utilize the interpolation inequality \eqref{interpol}, to bound $\int_{-\infty}^x |\bar u'(t)|^2\dd t$ by a constant multiplied by the left hand-side of \eqref{comput1bb}. Hence, $\theta(x)\leq c\theta'(x)$, $\forall x\leq x_0$, and for a constant $c>0$. Then we conclude as in Proposition \ref{estimate}.
\end{proof}

\bibliographystyle{plain}

\end{document}